\documentclass[11pt,twoside, leqno]{article}

\usepackage{amssymb,amsmath,amsfonts,amsthm,color,mathrsfs}
\usepackage[Symbol]{upgreek}
\usepackage{txfonts}
\usepackage[nottoc,notlot,notlof]{tocbibind}
\usepackage[active]{srcltx}
\usepackage{citeref}
\usepackage{hyperref}
\usepackage[dvipsnames, svgnames, x11names]{xcolor}

\usepackage{graphicx}
\usepackage{epsfig}

\usepackage{geometry}
\usepackage{graphicx}

\usepackage{subfig}
\usepackage{caption}

\usepackage{txfonts}
\allowdisplaybreaks
\def\rr{{\mathbb R}}
\def\rn{{{\rr}^n}}

\def\supp{{\mathop\mathrm{\,supp\,}}}

\newtheorem{thm}{Theorem}[section]
\newtheorem{lem}[thm]{Lemma}

\newtheorem{cor}[thm]{Corollary}

\newtheorem{rem}[thm]{Remark}

\numberwithin{equation}{section}

\begin{document}
\arraycolsep=1pt
\author{Zaihui Gan, Renjin Jiang \& Fanghua Lin}
\title{{\bf Improved Berezin-Li-Yau inequality and Kr\"oger inequality and consequences}
 \footnotetext{\hspace{-0.35cm} 2020 {\it Mathematics
Subject Classification}. Primary  35P15; Secondary  35P20, 42B37.
\endgraf{
{\it Key words and phrases: Dirichlet eigenvalue, Neumann eigenvalue, Berezin-Li-Yau inequality, Kr\"oger inequality}
\endgraf}}
\date{}}
\maketitle

\begin{center}
\begin{minipage}{11.5cm}\small
{\noindent{\bf Abstract}. We provide quantitative improvements to the Berezin-Li-Yau inequality and
the Kr\"oger inequality, in $\mathbb{R}^n$, $n\ge 2$. The improvement on Kr\"oger's inequality resolves an open question raised by Weidl from 2006.
The improvements allow us to show that, for any open bounded domains,
there are infinite many  Dirichlet eigenvalues satisfying P\'olya's conjecture if $n\ge 3$, and infinite many Neumann eigenvalues satisfying
P\'olya's conjecture if $n\ge 5$ and the Neumann spectrum is discrete.
}\end{minipage}
\end{center}
\vspace{0.2cm}
\tableofcontents

\section{Introduction}
\hskip\parindent
Given an open bounded domain $\Omega$ in $\rn$, $n\ge 2$. Let
$$0<\lambda_1(\Omega)\le \lambda_2(\Omega)\le \lambda_3(\Omega)\le \cdots$$
be eigenvalues to the Dirichlet Laplacian on $\Omega$.
When dealing with the Neumann Laplacian on $\Omega$, we shall assume that $-\Delta_N$ on $\Omega$
has a discrete spectrum as
$$0=\gamma_1(\Omega)\le\gamma_2(\Omega)\le\gamma_3(\Omega)\le \cdots.$$
We shall simply write $\lambda_k,\gamma_k$
instead of $\lambda_k(\Omega),\gamma_k(\Omega)$ when there is no confusion.

In 1954, P\'olya \cite{Polya54}  conjectured that, it should hold on arbitrary domain that
$$\frac{|\Omega|\omega(n)}{(2\pi)^n}\gamma_{k+1}^{n/2}\le k\le \frac{|\Omega|\omega(n)}{(2\pi)^n}\lambda_k^{n/2},$$
where $\omega(n)$ denotes the volume of unit ball in $\rn$.
P\'olya himself \cite{Polya61} proved this conjecture for  tiling domains (regular tiling domains for the Neumann case) in the plane.
His method extends to high dimensions, and so the conjecture is true for tiling domains in high dimensions too. See Kellner \cite{Kell66}
for removing the regular assumption regarding Neumann eigenvalues and simplification.

The P\'olya conjecture has been wide open since then. Other than tiling domains,
P\'olya's conjecture was known to be true on on product domains $\Omega_1\times\Omega_2\subset\rr^{n_1+n_2}$, $n_1\ge 2$,
$n_2\ge 1$ provided P\'olya's conjecture holds on $\Omega_1$,  by the work of Laptev \cite{La97}. Filonov et al. \cite{FLPS23,FLPS25}
verified P\'olya's conjecture on balls for Dirichlet eigenvalues for all $n\ge 2$, and for Neumann eigenvalues for $n=2$,
and on annuli  for Dirichlet eigenvalues for $n=2$. Recently, it was proved in \cite{JL25} that on a Lipschitz domain $\Omega$, for any $\epsilon\in (0,1)$ it holds for all
$\lambda_k\ge \Lambda(\Omega,\epsilon)$ that
$$k\le (1+\epsilon)\frac{|\Omega|\omega(n)}{(2\pi)^n}\lambda_k^{n/2},$$
where $\Lambda(\Omega,\epsilon)$ depending on the geometry of $\Omega$ is given explicitly. Moreover, \cite{JL25} also provides in all dimensions $n\ge 2$ classes of
irregular shaped domains that satisfy P\'olya's conjecture on the Dirichlet eigenvalue.

For general domains $\Omega$, it was known that P\'olya's conjecture holds for the first two Dirichlet eigenvalues and the first two non-trivial Neumann eigenvalues,
see  \cite{Henrot06}  for the Rayleigh-Faber-Krahn inequality ($\lambda_1(\Omega)$),   the Krahn-Szeg\"o inequality  ($\lambda_2(\Omega)$) and the
Szeg\"o-Weinberger inequality ($\gamma_2(\Omega)$),  and also recent breakthroughs   \cite{BuHe19,GNP09}  for $\gamma_3(\Omega)$.  In higher dimensions, $n\ge 3$, Freitas \cite{Fre19} proved that
the first $b(n)$ Dirichlet eigenvalues satisfying P\'olya conjecture, where $b(n)\ge 4$ and grows as dimension increases.

For arbitrary bounded open domains,
Berezin \cite{Berezin72} and  Li-Yau \cite{LY83} independently proved that
\begin{equation}\label{ly-sum}
\frac{ |\Omega|^{2/n}\omega(n)^{2/n}}{(2\pi)^2}\sum_{i=1}^k\lambda_i \ge \frac{n}{n+2} k^{\frac{n+2}{n}},
\end{equation}
which implies that
\begin{equation}\label{ly-eigenvalue}
\frac{ |\Omega|^{2/n}\omega(n)^{2/n}}{(2\pi)^2} \lambda_k \ge \frac{n}{n+2} k^{\frac{2}{n}}.
\end{equation}
Melas \cite{Me02} later gave a quantitative improvement of Li-Yau's estimate
as
\begin{equation}\label{melas-sum}
\sum_{i=1}^k\lambda_i \ge \frac{n}{n+2} \frac{(2\pi)^2}{ |\Omega|^{2/n}\omega(n)^{2/n}} k^{\frac{n+2}{n}}+C(n)k\frac{|\Omega|}{I(\Omega)},
\end{equation}
where $I(\Omega)$ is the moment of inertia of $\Omega$, i.e.,
\begin{align}\label{I-def}
I(\Omega):=\inf_{a\in\rn}\int_\Omega|x-a|^2\,dx.
\end{align}
Above $C(n)$ can be taken as $\frac{1}{24(n+2)}$, see Cheng et al. \cite[p. 36]{CQW13}, and a better constant for $n=2$, $C(2)=\frac{1}{32}$, from  \cite[(16)]{KVW09}.
Melas's result have been further improved by Kova\v{r}\'ik, Vugalter and Weidl \cite{KVW09}  for $n=2$,
where  a term of order $k^{\frac 32-\epsilon}$ can be added to the RHS of \eqref{melas-sum}  on domains with some smoothness.
See also \cite{GLW11,JX23,Weidl08} for further generalisations.

Kr\"oger \cite{Kroger92} proved for bounded domains with piecewisely smooth boundary that
\begin{align}\label{kroger-sum}
\frac{ |\Omega|^{2/n}\omega(n)^{2/n}}{(2\pi)^2}\sum_{j=1}^{k+1}\gamma_j \le \frac{n}{n+2} k^{1+2/n}.
\end{align}
and
\begin{align}\label{kroger-eigenvalue}
\frac{ |\Omega|^{2/n}\omega(n)^{2/n}}{(2\pi)^2}\gamma_{k+1} \le \left(\frac{n+2}{2}\right)^{2/n} k^{2/n}.
\end{align}
Li and Tang \cite{LT06} showed the above inequality \eqref{kroger-sum} is strict by multiplying on the right hand side an implicit factor $C(k,n,\Omega)<1$ which tends to 1 as $k$
goes to infinity.
We note Filonov \cite{Fil25} recently improved the estimate for Neumann eigenvalues on planar convex domains.

Laptev \cite{La97} showed that it holds for the Riesz means of Dirichlet eigenvalues and Neumann eigenvalues that
\begin{equation}\label{dirichlet-laptev}
\sum_{k:\,\lambda_k<\lambda} (\lambda-\lambda_k)\le \frac{2 }{n+2} \frac{\omega(n)|\Omega|}{(2\pi)^{n} } \lambda^{1+\frac{n}{2}}
\end{equation}
and
\begin{equation}\label{neumann-laptev}
\sum_{k:\,\gamma_k<\gamma} (\gamma-\gamma_k)\ge \frac{2 }{n+2} \frac{\omega(n)|\Omega|}{(2\pi)^{n} } \gamma^{1+\frac{n}{2}}.
\end{equation}
He then used these two inequalities to give a new proof of \eqref{ly-eigenvalue} and \eqref{kroger-eigenvalue}.
Note that Laptev \cite{La97} used the Riesz means to verify  P\'olya's conjecture on product domains $\Omega_1\times\Omega_2\subset\rr^{n_1+n_2}$
provided P\'olya's conjecture holds on $\Omega_1$ as we recalled above.

In recent works \cite{FL24a,FL24b}, Frank and Larson have launched the study of two term expansion of the Riesz means
for both Dirichlet/Neumann eigenvalues. In particular, asymptotic behavior have been established for Lipschitz domains
by  \cite{FL24a,FL24b}. Very recently Frank and Larson \cite{FL24a} obtained improvements of \eqref{dirichlet-laptev}
and \eqref{neumann-laptev} via the uncertainty principle, as
\begin{equation}\label{dirichlet-fl}
\sum_{k:\,\lambda_k<\lambda} (\lambda-\lambda_k)\le \frac{2 }{n+2} \frac{\omega(n)|\Omega|}{(2\pi)^{n} } \lambda^{1+\frac{n}{2}}\left(1-ce^{-c'(\Omega)\sqrt\lambda}\right)
\end{equation}
and
\begin{equation}\label{neumann-fl}
\sum_{k:\,\gamma_k<\gamma} (\gamma-\gamma_k)\ge \frac{2 }{n+2} \frac{\omega(n)|\Omega|}{(2\pi)^{n} } \gamma^{1+\frac{n}{2}}\left(1+ce^{-c'(\Omega)\sqrt\lambda}\right),
\end{equation}
where the constants $c,c'$ are implicit.

For a given bounded open domain $\Omega$, let
\begin{align}\label{J-def}
J_1:= \frac{ |\Omega|^{1/2}}{2|I(\Omega)|^{1/2}},
\end{align}
and $R_\Omega>0$ be such that $|\Omega|=|B(0,R_\Omega)|$.
Note that
\begin{align*}
\frac{n}{n+2}|\Omega|R_\Omega^2\le I(\Omega)=\inf_{a\in\rn}\int_\Omega|x-a|^2\,dx\le \frac{1}{4}(\mathrm{diam}(\Omega))^2|\Omega|,
\end{align*}
the quantity $J_1$ has  lower and upper  bounds as
\begin{align}\label{J-diameter}
\frac{1}{\mathrm{diam}(\Omega)}\le J_1\le  \sqrt{\frac{n+2}{n}}\frac{1}{2R_\Omega}.
\end{align}

Adopting the elegant method of using the moment of inertia by Males \cite{Me02}, we obtain the following estimates for the Riesz means of
Dirichlet eigenvalues.
\begin{thm}\label{Dirichlet-2}
Let $\Omega$ be a bounded open set in $\rn$, $n\ge 2$. Then  it holds for $\lambda>0$ that
\begin{align}\label{riesz-dirichlet}
\sum_{k:\,\lambda_k<\lambda}(\lambda-\lambda_k)&\le
\frac{2}{n+2} \lambda^{1+\frac n2} \frac{|\Omega| |\omega(n)|}{(2\pi)^{n}}\left(1 -C_1(n) \lambda^{-1}|J_1|^2\right)_+,
\end{align}
where $a_+=\max\{0,a\}$ and
\begin{align}\label{defn-c1}
C_1(n)=
\begin{cases}
\frac{157}{480}, \ & n=2,\\
\frac{5n(n+2)}{96}\left(1-\frac{\sqrt{n+2}}{8\sqrt n \Gamma(1+\frac n2))^{2/n}}\right)^{n-1}\left(2-\frac{\sqrt{n+2}}{8\sqrt n \Gamma(1+\frac n2))^{2/n}}\right), \ & n\ge 3.
\end{cases}
\end{align}
\end{thm}
Theorem \ref{Dirichlet-2} will follow from a sharper estimate below (Theorem \ref{Dirichlet-1}). With the aid of  Stirling's formula we see that in $C_1(n)$
the  factor involving Gamma function have a uniformly positive lower bound, about 1.014,  for all $n\ge 3$, see Remark \ref{remark-1} below.
Theorem \ref{Dirichlet-2} also implies an eigenvalue estimate (cf. Corollary \ref{eigenvalue-dirichlet-1} below),
which shows that the constant $C_1(n)$ improves on previous version of Melas' improvement \eqref{melas-sum}.

Recall that Weyl's conjecture on sharper asymptotic behavior of Dirichlet eigenvalue states for domains with piecewise smooth boundary that
\begin{align}\label{conj-weyl}
\mathcal{N}^D_{\Omega}(\lambda)=\frac{|\Omega|\omega(n)}{(2\pi)^{n}} \lambda^{n/2}-\frac{\mathcal{H}^{n-1}(\partial\Omega)}{2^{n+1} \pi^{\frac{d-1}{2}}\Gamma(\frac{n+1}{2})}\lambda^{\frac{n-1}{2}}+o(\lambda^{\frac{n-1}2}),
\end{align}
where  $\mathcal{N}^D_{\Omega}(\lambda)$ is the counting function
$$\mathcal{N}^D_{\Omega}(\lambda):=\#\{\lambda_k(\Omega):\,\lambda_k(\Omega)<\lambda\}.$$
The conjecture has been proved by Ivrii \cite{Ivr80} under the condition that the set of all periodic geodesic billiards in $\Omega$ has measure zero, and proved in \cite{SV97} for convex analytic domains
and polygons.
From \eqref{conj-weyl}, one might hope to replace $(1 -C_1(n) \lambda^{-1}|J_1|^2)$  in R.H.S. of \eqref{riesz-dirichlet}
by something like $(1 -C_1(n) \lambda^{-1/2}|J_1|)$. However, since we didn't assume any regularity of $\partial\Omega$, $\mathcal{H}^{n-1}(\partial\Omega)$
might be infinite,  Weyl's sharper asymptotic formula cannot hold in such general settings, see however \cite[Theorem 2]{KVW09}
for domains with certain smoothness on $\rr^2$.

Theorem \ref{Dirichlet-2} implies the following qualitative result.
\begin{cor}\label{cor-infinite-polya}
Let $\Omega$ be a bounded open set in $\rn$, $n\ge 3$. Then there exist infinitely many eigenvalues $\{\lambda_{k_j}\}_{j\in\mathbb{N}}$ satisfying P\'olya's conjecture, i.e.,
\begin{align}
\frac{|\Omega| |\omega(n)|}{(2\pi)^{n}}\lambda_{k_j}^{\frac n2}\ge k_j.
\end{align}
\end{cor}
The corollary will follow from Theorem \ref{infinite-dirichlet-polya} below, where we in fact show that
for a certain length of interval there is at least one eigenvalue satisfies P\'olya's conjecture.


We have a similar but weaker result for the Neumann case which improves the Kr\"oger inequality from \cite{Kroger92}, quantitatively.
Note that,  it is an open question raised by Weidl from 2006\footnote[1]{Problems from the Workshops on
Low Eigenvalues  of Laplace and Schr\"odinger Operators
at AIM 2006 (Palo Alto) and MFO 2009 (Oberwolfach), https://aimath.org/WWN/loweigenvalues/loweigenvalues.pdf.},
that ``Can one strengthen the Kr\"oger result by including a correction term?" We note that previously \cite{FL24a,LT06}
had provided some improvements but without explicit constants.
\begin{thm}\label{Neumann}
Let $\Omega$ be a bounded open set in $\rn$, $n\ge 2$. Suppose that $-\Delta_N$ has a discrete spectrum
$$0= \gamma_1\le \gamma_2\le \gamma_3\le \cdots.$$ Then
it holds  for $\gamma>0$ that
\begin{align}\label{riesz-neumann}
\sum_{k:\,\gamma_k<\gamma} (\gamma-\gamma_k)&\ge  \frac{\omega(n)|\Omega|}{(2\pi)^{n}} \frac{2\gamma^{1+\frac n2}}{n+2}\left(1+\frac{n(n+2) }{3 } \frac{J_1^4}{\gamma^{1/2} (4J_1^2+\gamma)^{3/2}}\right).
\end{align}
\end{thm}
Theorem \ref{Neumann} also implies a qualitative result as following, see Theorem \ref{thm-infinite-neumann} below for a more quantitative statement.
\begin{cor}\label{cor-infinite-neumann}
Let $\Omega$ be a bounded open set in $\rn$, $n\ge 5$. Suppose that $-\Delta_N$ has a discrete spectrum
$$0= \gamma_1\le \gamma_2\le \gamma_3\le \cdots.$$
Then there exist infinitely many eigenvalues $\{\gamma_{k_j}\}_{j\in\mathbb{N}}$ satisfying P\'olya's conjecture, i.e.,
\begin{align}
 k_{j}- 1\ge \frac{\omega(n)|\Omega|}{(2\pi)^{n}} \gamma_{k_j}^{n/2}.
\end{align}
\end{cor}
From the technique point of view the improvement on Neumann case is weaker seems somehow nature to us, since
for the Dirichlet case, the quantity $\sum_{k:\,\lambda_k<\lambda}(\lambda-\lambda_k)$ equals to the integral of the function
$$C(\lambda-|\xi|^2)\sum_{k:\,\lambda_k<\lambda}|\hat{\phi_k}(\xi)|^2$$
on $\rn$,
while for the Neumann case, the Fourier transform does not simply transform the Laplacian to polynomial multipliers.

Let us note that one may also use the improved Berezin-Li-Yau inequality and Kr\"oger inequality to give refined eigenvalue estimates
on product domains, see Laptev \cite{La97} and also \cite{JL25,Larson17}.

The paper is organized as follows.

In Section 2, we deal with the improved Berezin-Li-Yau inequality, prove Theorem \ref{Dirichlet-2}
and prove
a more quantitative result which implies Corollary \ref{cor-infinite-polya}.

In Section 3, we prove the improved Kr\"oger inequality (Theorem \ref{Neumann}) and prove
a more quantitative result which implies  Corollary \ref{cor-infinite-neumann}.

%
%
%
%
%

\section{Improved Berezin-Li-Yau inequality}
\hskip\parindent In this section we prove Theorem \ref{Dirichlet-2} and Corollary \ref{cor-infinite-polya}, which will follow from a sharper version, Theorem \ref{Dirichlet-1} below.
In this section, we assume that $n\ge 2$.

For $x\in\Omega$, let
$$f_\lambda(x)=\sum_{k:\,\lambda_k<\lambda} \phi_k(x)^2,$$
where $\phi_k$ is the corresponding normalized eigenfunction to $\lambda_k$, and for $\xi\in\rn$, let
$$F_\lambda(\xi)=\sum_{k:\,\lambda_k<\lambda} |\hat{\phi}_k(\xi)|^2.$$
It is well known by using the orthogonal property of eigenfunctions $\{\phi_k\}$ and viewing $\hat{\phi}_k(\xi)$ as
the function $(2\pi)^{-n/2}e^{ix\cdot\xi}$ projecting to the bases $\{\phi_k\}_{\lambda_k<\lambda}$ (see \cite{LY83} for instance) that
$$F_\lambda(\xi)=\sum_{k:\,\lambda_k<\lambda} |\hat{\phi}_k(\xi)|^2\le (2\pi)^{-n}|\Omega|.$$

On the other hand, by using  the function $(2\pi)^{-n/2} xe^{ix\cdot\xi}$ projecting to the bases $\{\phi_k\}_{\lambda_k<\lambda}$  we have
\begin{align}
\sum_{k:\,\lambda_k<\lambda} |\nabla\hat{\phi}_k(\xi)|^2&\le \sum_{k:\,\lambda_k<\lambda} (2\pi)^{-n}\left|\int_\Omega ixe^{-ix\xi}\phi_k(x)\,dx\right|^2\le (2\pi)^{-n}\int_\Omega|x|^2\,dx.
\end{align}
In what follows,
we denote
$$I(\Omega)=\inf_{a\in\rn}\int_\Omega|x-a|^2\,dx.$$
By translation, we see that
\begin{align}
\sum_{k:\,\lambda_k<\lambda} |\nabla\hat{\phi}_k(\xi)|^2&\le (2\pi)^{-n}I(\Omega).
\end{align}
Note that
\begin{align}
I(\Omega)\le \frac{1}{4} \mathrm{diam}(\Omega)^2 |\Omega|,
\end{align}
and by the Hardy-Littlewood inequality (cf. \cite[Corollary 1.4.1]{Kesa06}), it holds that
\begin{align}\label{hl-inequality}
I(\Omega)\ge \int_{B(0,R_\Omega):|B(0,R_\Omega)|=|\Omega|}|x|^2\,dx=\frac{n}{n+2}R_\Omega^2|\Omega|.
\end{align}
The H\"older inequality then implies that
\begin{align}\label{gradient-f}
|\nabla F_\lambda|\le 2 \sum_{k:\,\lambda_k<\lambda} |\hat{\phi}_k(\xi)||\nabla \hat{\phi}_k(\xi) |\le 2(2\pi)^{-n}|\Omega|^{1/2}|I(\Omega)|^{1/2}.
\end{align}

For $t>0$, we define the non-increasing re-arrangement of $F_\lambda(\xi)$ as
$$F_\lambda^\#(t):=
\begin{cases}
\sup_{\xi\in\rn}F_\lambda(\xi), & t=0,\\
\inf\{s:\,|\{\xi\in\rn:\,F_\lambda(\xi)>s\}|\le t\}, &t>0.
\end{cases}$$
The Schwarz symmetrization of  $F_\lambda(\xi)$ is then defined as
$$F_\lambda^\ast(y)=F_\lambda^\#(\omega(n)|y|^n).$$

\begin{lem}\label{lem-1}
The function $F_\lambda^\ast\in W^{1,p}(\rn)$ for all $p\in [1,\infty]$, and in particular,
$$\|\nabla F_\lambda^\ast\|_{L^p(\rn)}\le \|\nabla F_\lambda\|_{L^p(\rn)}, \ \forall\, p\in [1,\infty].$$
\end{lem}
\begin{proof}
By definition of $F_\lambda$, we have
\begin{align*}
|\nabla F_\lambda(\xi)|\le 2 \sum_{k:\,\lambda_k<\lambda} |\hat{\phi}_k(\xi)||\nabla \hat{\phi}_k(\xi) |\le 2F_\lambda(\xi)^{1/2} \left( \sum_{k:\,\lambda_k<\lambda} (2\pi)^{-n}\left|\int_\Omega ixe^{-ix\xi}\phi_k(x)\,dx\right|^2\right)^{1/2}.
\end{align*}
Since $\phi_k$ and $x\phi_k(x)$ belong to $L^1(\rn)\cap L^\infty(\rn)$, the Plancherel principle and the fact Fourier transform maps $L^1(\rn)$ to $L^\infty(\rn)$ imply that
\begin{align*}
|\nabla F_\lambda(\xi)|\in L^1(\rn)\cap L^\infty(\rn).
\end{align*}
The P\'olya-Szeg\H{o} inequality then implies  for all $p\in [1,\infty)$ that
$$\|\nabla F_\lambda^\ast\|_{L^p(\rn)}\le \|\nabla F_\lambda\|_{L^p(\rn)}.$$

For the case $p=\infty$, we can follow the proof of \cite[Theorem 2.3.2]{Kesa06} in the case of bounded domains.
From the cases $p<\infty$, it follows that $F_\lambda^\ast\in W^{1,p}(\rn)$ for all $p\in [1,\infty)$, and then from the Sobolev embedding that
$F_\lambda^\ast$ is in the class $C^\alpha(\rn)$ for all $\alpha<1$. We further see that
$F_\lambda^\#$ is continuous on $(0,\infty)$.

For $s>0$, set $t=F_\lambda^\#(s)$. Let
$$F_\lambda^\#(s+k)=t-h.$$
Then it holds that
\begin{align*}
\int_{\{t-h<F_\lambda\le t\}} |\nabla F_\lambda|\,dx \le \|\nabla F_\lambda\|_{L^\infty(\rn)} \left(\mu(t-h)-\mu(t)\right),
\end{align*}
where
$$\mu(t)=|\{\xi\in\rn:\,F_\lambda(\xi)>t\}|$$
is the distribution of $F_\lambda$.
The coarea formula together with the isoperimetric inequality then implies that
\begin{align*}
\int_{\{t-h<F_\lambda\le t\}} |\nabla F_\lambda|\,dx &=\int_{t-h}^t P(\{\xi\in\rn:\,F_\lambda(\xi)>\tau\})\,d\tau\\
&\ge n\omega(n)^{1/n}  \int_{t-h}^t  \mu(\tau)^{1-1/n}\,d\tau\\
&\ge n\omega(n)^{1/n}h \mu(t)^{1-1/n}.
\end{align*}
This further implies that
\begin{align*}
F_\lambda^\#(s)-F_\lambda^\#(s+k)&\le \frac{\|\nabla F_\lambda\|_{L^\infty(\rn)}}{n\omega(n)^{1/n}} ks^{-1+1/n}
\end{align*}
and hence,
\begin{align*}
0<-\frac{\,dF_\lambda^\#(s)}{\,ds}\le \frac{\|\nabla F_\lambda\|_{L^\infty(\rn)}}{n\omega(n)^{1/n}} s^{-1+1/n}.
\end{align*}
It follows then for $|x|>|y|$, $F_\lambda^\ast(x)\le F_\lambda^\ast(y)$ and
\begin{align*}
F_\lambda^\ast(y)-F_\lambda^\ast(x)=\int_{\omega(n)|y|^n}^{\omega(n)|x|^n}-\frac{\,dF_\lambda^\#(s)}{\,ds}\,ds\le
\|\nabla F_\lambda\|_{L^\infty(\rn)}(|x|-|y|).
\end{align*}
The proof is complete.
\end{proof}

For simplicity of notions we denote by  $\psi(r)=\psi(|x|)=F_\lambda^\ast(x)$.
\begin{lem}\label{lem-2}
It holds that
\begin{align*}
\int_{\rn}(\lambda-|\xi|^2)F_\lambda(\xi)\,d\xi\le
\int_0^\infty n\omega(n) \left(\frac{\lambda r^n}{n} -\frac{r^{n+2}}{n+2}\right) (-\psi'(r))\,dr.
\end{align*}
\end{lem}
\begin{proof}
We first truncate the function $|\xi|^2$
to
$$
G_k(\xi)=
\begin{cases} |\xi|^2, & |\xi|\le k\\
k^2, & |\xi|>k.
\end{cases}
$$
Then $G_k$ is in $L^\infty(\rn)$ and is radially non-decreasing.

Since $F_\lambda^\ast(x)$ and $F_\lambda(x)$ are equi-measurable and belong to $L^1(\rn)$, $F_\lambda^\ast(x)$ is radially non-increasing, $G_k(\xi)$ is radially non-decreasing and belongs to $L^\infty(\rn)$, we deduce from
the Hardy-Littlewood inequality (cf. \cite[Corollary 1.4.1]{Kesa06}) that
\begin{align*}
\int_{\rn}|\xi|^2F_\lambda(\xi)\,d\xi\ge \int_{\rn}G_k(\xi)F_\lambda(\xi)\,d\xi\ge \int_{\rn}G_k(\xi)F_\lambda^\ast(\xi)\,d\xi.
\end{align*}
Letting $k\to \infty$ and using the monotone convergence theorem give that
\begin{align*}
\int_{\rn}|\xi|^2F_\lambda(\xi)\,d\xi\ge \int_{\rn}|\xi|^2F_\lambda^\ast(\xi)\,d\xi.
\end{align*}

We therefore have that
\begin{align}\label{e2.6}
\int_{\rn}(\lambda-|\xi|^2)F_\lambda(\xi)\,d\xi&\le \int_{\rn}(\lambda-|\xi|^2)F_\lambda^\ast(\xi)\,d\xi\nonumber\\
&=\int_{0}^\infty(\lambda-r^2)\psi(r)n\omega(n)r^{n-1}\,dr.
\end{align}

Since $\psi(|x|)=F_\lambda^\ast(x)$ is bounded, non-increasing on $(0,\infty)$,
$$\int_{0}^\infty\psi(r)n\omega(n)r^{n-1}\,dr=\int_{\rn}F_\lambda(\xi)\,d\xi<\infty,$$
and
$$\int_{0}^\infty\psi(r) \omega(n){n} r^{n+1}\,dr\le \int_{\rn}|\xi|^2F_\lambda(\xi)\,d\xi<\infty,$$
we see that
\begin{align*}
\lim_{r\to\infty} \psi(r)r^{n+2}= 0.
\end{align*}
Therefore for large enough $T$, saying $T>200\sqrt\lambda$,  we integrate by parts and obtain
\begin{align*}
\int_{0}^T(\lambda-r^2)\psi(r)n\omega(n)r^{n-1}\,dr&= {\omega(n)}{n}\left(\frac{\lambda r^n}{n}-\frac{r^{n+2}}{n+2}\right) \psi(r)\bigg|_{0}^{T}- \int_0^T n\omega(n)\left(\frac{\lambda r^n}{n} -\frac{r^{n+2}}{n+2}\right) (\psi'(r))\,dr\\
&= {\omega(n)}{n}\left(\frac{\lambda T^n}{n}-\frac{T^{n+2}}{n+2}\right) \psi(T)- \int_0^T n\omega(n)\left(\frac{\lambda r^n}{n} -\frac{r^{n+2}}{n+2}\right) (\psi'(r))\,dr,
\end{align*}
which is equivalent to
\begin{align*}
\int_0^T n\omega(n)\left(\frac{\lambda r^n}{n} -\frac{r^{n+2}}{n+2}\right) (\psi'(r))\,dr&={\omega(n)}{n}\left(\frac{\lambda T^n}{n}-\frac{T^{n+2}}{n+2}\right) \psi(T)-\int_{0}^T(\lambda-r^2)\psi(r)n\omega(n)r^{n-1}\,dr.
\end{align*}
Since $\lim_{r\to\infty} \psi(r)r^{n+2}= 0$ (so $\lim_{r\to\infty} \psi(r)r^{n}=0$ too),
$$\left |\int_{0}^T(\lambda-r^2)\psi(r)n\omega(n)r^{n-1}\,dr \right|\le \int_{0}^\infty(\lambda+r^2)\psi(r)n\omega(n)r^{n-1}\,dr  <\infty,$$
and the integral $\int_0^T n\omega(n)\left(\frac{\lambda r^n}{n} -\frac{r^{n+2}}{n+2}\right) (\psi'(r))\,dr$ is increasing on $T\in (\sqrt{\frac{n+2}{n}\lambda},\infty)$,
we see that $\int_0^T n\omega(n)\left(\frac{\lambda r^n}{n} -\frac{r^{n+2}}{n+2}\right) (\psi'(r))\,dr$, as a function in $T$, is increasing on $(\sqrt{\frac{n+2}{n}\lambda},\infty)$
and uniformly bounded from above.
Therefore the monotone converges theorem guarantees that
\begin{align*}
\int_{0}^\infty(\lambda-r^2)\psi(r)n\omega(n)r^{n-1}\,dr&= {\omega(n)}{n}\left(\frac{\lambda r^n}{n}-\frac{r^{n+2}}{n+2}\right) \psi(r)\bigg|_{0}^{\infty}- \int_0^\infty n\omega(n)\left(\frac{\lambda r^n}{n} -\frac{r^{n+2}}{n+2}\right) (\psi'(r))\,dr\\
&=- \int_0^\infty n\omega(n)\left(\frac{\lambda r^n}{n} -\frac{r^{n+2}}{n+2}\right) (\psi'(r))\,dr.
\end{align*}
This together with \eqref{e2.6} completes the proof.
\end{proof}

\begin{lem}\label{lem-3}
Let $n\ge 2$ and $\lambda>0$. Suppose that $g\ge 0$ on $(0,\infty)$, which satisfies
$$\int_0^\infty g(r)\,dr\le A,\, g(r)\le B$$
with $A,B>0$ satisfying $A/B<\sqrt{\frac{n+2}{n}\lambda}$.
Then
$$\int_0^\infty \left(\frac{\lambda r^n}{n} -\frac{r^{n+2}}{n+2}\right)g(r)\,dr\le B\left(\frac{\lambda \left[(t+A/B)^{n+1}-t^{n+1}\right]}{n(n+1)}-\frac{\left[(t+A/B)^{n+3}-t^{n+3}\right]}{(n+2)(n+3)}\right),$$
where  $t<\sqrt\lambda$  is such that $t+A/B>\sqrt{\lambda}$ and
\begin{align*}
\frac{\lambda \left[(t+A/B)^{n}-t^{n}\right]}{n}=\frac{\left[(t+A/B)^{n+2}-t^{n+2}\right]}{(n+2)}.
\end{align*}
\end{lem}
\begin{proof}
Note that the function
$$\frac{\lambda r^n}{n} -\frac{r^{n+2}}{n+2}$$
tends to zero when $r\to 0$, and is minus infinite when $r=\infty$.
Moreover, there is only one maximum in $(0,\infty)$, which attains at the point
$$r=\sqrt\lambda,$$
and the function is negative if
$$\lambda< \frac{n}{n+2}r^2\Longleftrightarrow r>\sqrt{\frac{n+2}{n}\lambda}.$$
So for any $g$ satisfies the assumption,
$$\int_0^\infty \left(\frac{\lambda r^n}{n} -\frac{r^{n+2}}{n+2}\right)g(r)\,dr\le \int_0^{\sqrt{\frac{n+2}{n}\lambda}} \left(\frac{\lambda r^n}{n} -\frac{r^{n+2}}{n+2}\right)g(r)\,dr.$$
Therefore we may and do assume that $\supp g\subset (0,\sqrt{\frac{n+2}{n}\lambda})$.

Further, note that $\sqrt{\frac{n+2}{n}\lambda}-A/B>0$.
If $\int_0^{\sqrt{\frac{n+2}{n}\lambda}}g\,dr<A$, we can find another $\tilde g$ such that
$B\ge \tilde g\ge g$ and $\int_0^{\sqrt{\frac{n+2}{n}\lambda}}\tilde g\,dr=A$. It is obvious that
$$\int_0^{\sqrt{\frac{n+2}{n}\lambda}} \left(\frac{\lambda r^n}{n} -\frac{r^{n+2}}{n+2}\right)g(r)\,dr\le \int_0^{\sqrt{\frac{n+2}{n}\lambda}} \left(\frac{\lambda r^n}{n} -\frac{r^{n+2}}{n+2}\right)\tilde g(r)\,dr.$$

For $t\in (0,\sqrt{\frac{n+2}{n}\lambda}-A/B )$, let us set
$$g_t(r)=
\begin{cases}
B, & r\in (t,t+A/B)\\
0, & \mbox{otherwise}.
\end{cases}
$$
Then
\begin{align*}
\int_0^{\sqrt{\frac{n+2}{n}\lambda}} \left(\frac{\lambda r^n}{n} -\frac{r^{n+2}}{n+2}\right)g_t(r)\,dr&=B\left(\frac{\lambda r^{n+1}}{n(n+1)}-\frac{r^{n+3}}{(n+2)(n+3)}\right)\bigg|_{t}^{t+A/B}\\
&=B\left(\frac{\lambda \left[(t+A/B)^{n+1}-t^{n+1}\right]}{n(n+1)}-\frac{\left[(t+A/B)^{n+3}-t^{n+3}\right]}{(n+2)(n+3)}\right)\\
&=:B G(t).
\end{align*}
The last term $G(t)$ as a function of $t$, converges to $-\infty$ as $t\to\infty$, and when $t=0$, it holds that
\begin{align*}
\left(\frac{\lambda \left[(A/B)^{n+1}\right]}{n(n+1)}-\frac{\left[(A/B)^{n+3}\right]}{(n+2)(n+3)}\right)& =(A/B)^{n+1} \left(\frac{\lambda }{n(n+1)}-\frac{\left[(A/B)^{2}\right]}{(n+2)(n+3)}\right)\\
&\ge (A/B)^{n+1} \left(\frac{\lambda }{n(n+1)}-\frac{\lambda}{n(n+3)}\right)\\
&= (A/B)^{n+1} \frac{2\lambda }{n(n+1)(n+3)}.
\end{align*}
The maximal point of the $G(t)$ is attained at the point $t_0$ such that
\begin{align*}
\frac{\lambda \left[(t_0+A/B)^{n}-t_0^{n}\right]}{n}=\frac{\left[(t_0+A/B)^{n+2}-t_0^{n+2}\right]}{(n+2)}.
\end{align*}
Equivalently,
\begin{align*}
\frac{\lambda(t_0+A/B)^{n}}{n}-\frac{\lambda(t_0+A/B)^{n+2}}{n+2} =\frac{\lambda t_0^{n}}{n}-\frac{t_0^{n+2}}{(n+2)}.
\end{align*}
Since $\frac{\lambda r^n}{n} -\frac{r^{n+2}}{n+2}$ is increasing on $(0,\sqrt\lambda)$ and decreasing on $(\sqrt\lambda,\infty)$, we see that
$t_0<\sqrt\lambda$ and $t_0+A/B>\sqrt{\lambda}$.

For any other $\supp g\subset (0,\sqrt{\frac{n+2}{n}\lambda})$ satisfying $\int_0^{\sqrt{\frac{n+2}{n}\lambda}} g\,dr=A$ and the assumptions,
suppose that $g$ is not fully supported in $[t_0,t_0+A/B]$, i.e.,
$$\int_{t_0}^{t_0+A/B} g(t)\,dt<A.$$
Then
\begin{align*}
\int_{(0,\sqrt{\frac{n+2}{n}\lambda})\setminus [t_0,t_0+A/B]} \left(\frac{\lambda r^n}{n} -\frac{r^{n+2}}{n+2}\right)g(r)\,dr&\le
\sup_{t\in (0,\sqrt{\frac{n+2}{n}\lambda})\setminus [t_0,t_0+A/B]}\left(\frac{\lambda r^n}{n} -\frac{r^{n+2}}{n+2}\right) \left(A-\int_{t_0}^{t_0+A/B} g(t)\,dt\right) \\
&=\inf_{t\in [t_0,t_0+A/B]}\left(\frac{\lambda r^n}{n} -\frac{r^{n+2}}{n+2}\right) \left(\int_{t_0}^{t_0+A/B} \left(g_{t_0}(t)-g(t)\right)\,dt\right)\\
&\le \int_{t_0}^{t_0+A/B} \left(\frac{\lambda r^n}{n} -\frac{r^{n+2}}{n+2}\right)\left(g_{t_0}(t)-g(t)\right)\,dt.
\end{align*}
The proof is complete.
\end{proof}

\begin{thm}\label{Dirichlet-1}
Let $\Omega$ be a bounded open set in $\rn$, $n\ge 2$. Then it holds for $\lambda>0$ that
\begin{align*}
\sum_{k:\,\lambda_k<\lambda}(\lambda-\lambda_k)
\begin{cases}
=0, & \lambda\le \frac{(2\pi)^2}{|\Omega|^{2/n}\omega(n)^{2/n}},\\
\le   \frac{2n\omega(n)}{(2\pi)^{n}} |\Omega|^{1/2}|I(\Omega)|^{1/2}   \left[\frac{\lambda ((t_0+J_1)^{n+1}-t_0^{n+1})}{n(n+1)}-\frac{(t_0+J_1)^{n+3}-t_0^{n+3}}{(n+2)(n+3)}\right], & \lambda> \frac{(2\pi)^2}{|\Omega|^{2/n}\omega(n)^{2/n}},
\end{cases}
\end{align*}
where  $t_0<\sqrt\lambda$  is such that $t_0+J_1>\sqrt{\lambda}$ and
\begin{align*}
\frac{\lambda \left[(t_0+J_1)^{n}-t_0^{n}\right]}{n}=\frac{\left[(t_0+J_1)^{n+2}-t_0^{n+2}\right]}{(n+2)}.
\end{align*}
\end{thm}
\begin{proof}
By the definition of $F_\lambda$ and Lemma \ref{lem-2}, we have
\begin{align}\label{est-1}
\sum_{k:\,\lambda_k<\lambda}(\lambda-\lambda_k)&= \sum_{k:\,\lambda_k<\lambda}\int_{\Omega} (\lambda-\lambda_k)|\phi_k(x)|^2\,dx \nonumber\\
&= \sum_{k:\,\lambda_k<\lambda}\int_{\rn} (\lambda-|\xi|^2)|\hat{\phi}_k(\xi)|^2\,d\xi \nonumber\\
&= \int_{\rn} (\lambda-|\xi|^2) F_\lambda(\xi)\,d\xi\nonumber\\
&\le \int_0^\infty n\omega(n)\left(\frac{\lambda r^n}{n} -\frac{r^{n+2}}{n+2}\right) (-\psi'(r))\,dr,
\end{align}
where $\psi(r)=\psi(|x|)=F_\lambda^\ast(x)$.

By Lemma \ref{lem-1} and \eqref{gradient-f}, we have that
$$|\psi'(|x|)|=|\nabla F_\lambda^\ast(x)|\le \|\nabla F_\lambda\|_{L^\infty(\rn)}\le 2(2\pi)^{-n}|\Omega|^{1/2}|I(\Omega)|^{1/2}.$$

On the other hand, we have
$$0\le \psi(r)\le (2\pi)^{-n} |\Omega|.$$

Now set $A=(2\pi)^{-n} |\Omega|$ and $B=2(2\pi)^{-n}|\Omega|^{1/2}|I(\Omega)|^{1/2}$ in Lemma \ref{lem-3}.
Recall from \eqref{hl-inequality} that
\begin{align*}
I(\Omega)\ge \int_{B(0,R_\Omega)}|x|^2\,dx=\frac{n}{n+2}R_\Omega^2|\Omega|,
\end{align*}
and we further have
\begin{align}\label{J-compare}
A/B=\frac{(2\pi)^{-n} |\Omega|}{2(2\pi)^{-n}|\Omega|^{1/2}|I(\Omega)|^{1/2}}=\frac{|\Omega|^{1/2}}{2|I(\Omega)|^{1/2}}=J_1\le \frac{\sqrt{n+2}}{2R_\Omega\sqrt n}.
\end{align}
where $J_1$ is defined by \eqref{J-def}, $R_\Omega>0$ is such that $|B(0,R_\Omega)|=|\Omega|.$

On the other hand, note that for the first eigenvalue we always have
\begin{equation}\label{first-eigenvalue}
\lambda_1 \ge  \frac{(2\pi)^2}{|\Omega|^{2/n}\omega(n)^{2/n}}=\frac{(2\pi)^2}{R_\Omega^2\omega(n)^{4/n}}.
\end{equation}
Therefore for $\lambda>\frac{(2\pi)^2}{|\Omega|^{2/n}\omega(n)^{2/n}}$, it holds that
\begin{align}\label{first-j1}
\sqrt{\frac{n+2}{n}\lambda}& >\sqrt{\frac{n+2}{n}\lambda_1} \ge \sqrt{\frac{n+2}{n}} (2\pi)  |\omega(n)|^{-2/n} R_\Omega^{-1}\nonumber\\
&\ge 4(\Gamma(1+\frac n2))^{2/n}J_1=4(\Gamma(1+\frac n2))^{2/n} A/B,
\end{align}
since
$$\omega(n)=\frac{\pi^{n/2}}{\Gamma(1+\frac n2)}.$$

We therefore deduce from Lemma \ref{lem-3} that
\begin{align}\label{est-2}
\sum_{k:\,\lambda_k<\lambda}(\lambda-\lambda_k)&\le \int_0^\infty n\omega(n)\left(\frac{\lambda r^n}{n} -\frac{r^{n+2}}{n+2}\right) (-\psi'(r))\,dr\nonumber\\
&\le  \frac{2n\omega(n)}{(2\pi)^{n}} |\Omega|^{1/2}|I(\Omega)|^{1/2}    \left[\frac{\lambda ((t_0+J_1)^{n+1}-t_0^{n+1})}{n(n+1)}-\frac{(t_0+J_1)^{n+3}-t_0^{n+3}}{(n+2)(n+3)}\right],
\end{align}
where  $t_0<\sqrt\lambda$ is such that  $t_0+A/B>\sqrt{\lambda}$ and
\begin{align*}
\frac{\lambda \left[(t_0+J_1)^{n}-t_0^{n}\right]}{n}=\frac{\left[(t_0+J_1)^{n+2}-t_0^{n+2}\right]}{(n+2)},
\end{align*}
the proof is complete.
\end{proof}
\begin{rem}\label{remark-improvement}\rm
Previous Berezin-Li-Yau inequality
\begin{align*}
\sum_{k:\,\lambda_k<\lambda}(\lambda-\lambda_k)&\le  \frac{2}{n+2} \frac{\omega(n)|\Omega|}{(2\pi)^{n}}\lambda^{1+n/2}
\end{align*}
can simply be obtained  by letting $(-\psi'(r))$ in \eqref{est-2} be multiples of the Dirac function at the point $r=\sqrt\lambda$, i.e.,
$$(-\psi'(r))=\frac{|\Omega|}{(2\pi)^n}\delta_{\sqrt\lambda}(r).$$
Recall that $r=\sqrt\lambda$ is the {\em only} maximal point of the function
$$\frac{\lambda r^n}{n} -\frac{r^{n+2}}{n+2}.$$
For our result, the realization function is in the form
$$
-\psi'(r)=\begin{cases}
2(2\pi)^{-n}|\Omega|^{1/2}|I(\Omega)|^{1/2},&  r\in [t_0,t_0+J_1]\\
0, & otherwise.
\end{cases}$$
Note that
$$\int_{t_0}^{t_0+J_1} (-\psi'(r))\,dr=\frac{|\Omega|}{(2\pi)^n}.$$
From this one sees that Theorem \ref{Dirichlet-1} improves
the original Berezin-Li-Yau inequality in an obvious way.
\end{rem}

We turn to prove Theorem \ref{Dirichlet-2}.
\begin{proof}[Proof of Theorem \ref{Dirichlet-2}]
We only need to consider the case $\lambda>\frac{(2\pi)^2}{|\Omega|^{2/n}\omega(n)^{2/n}}$.

If $n=2$,
$$\frac{\lambda (n+2)}{n} [(t_0+J_1)^{n}-t_0^{n}] = (t_0+J_1)^{n+2}-t_0^{n+2}$$
gives that
$$2\lambda= (t_0+J_1)^{2}+t_0^{2},$$
and
$$t_0=\sqrt{\lambda-\frac{J_1^2}{4}}-\frac{J_1}{2}.$$
So
\begin{align*}
 &\frac{\lambda ((t_0+J_1)^{n+1}-t_0^{n+1})}{n(n+1)}-\frac{(t_0+J_1)^{n+3}-t_0^{n+3}}{(n+2)(n+3)}\nonumber\\
  &=\frac{2\lambda \left[(\sqrt{\lambda-\frac{J_1^2}{4}}+\frac{J_1}{2})^{3}-(\sqrt{\lambda-\frac{J_1^2}{4}}-\frac{J_1}{2})^{3} \right] }{30}+\frac{\lambda J_1 (\sqrt{\lambda-\frac{J_1^2}{4}}-\frac{J_1}{2})^{2} }{{10}}-\frac{J_1  (\sqrt{\lambda-\frac{J_1^2}{4}}-\frac{J_1}{2})^{4}}{20}\\
  &=\frac{ 2\lambda J_1 \left[ 2\lambda+\lambda -\frac{J_1^2}{4}-\frac{J_1^2}{4}\right]  }{30}+\frac{J_1 (\sqrt{\lambda-\frac{J_1^2}{4}}-\frac{J_1}{2})^{2}  (2\lambda- \lambda+ J_1\sqrt{\lambda-\frac{J_1^2}{4}})}{20}\\
  &= \frac{ 2\lambda J_1 \left[ 3\lambda -\frac{J_1^2}{2}\right]  }{30}+\frac{J_1 ( \lambda^2-J_1^2 (\lambda-\frac{J_1^2}{4}))}{20}\\
  &=\frac{\lambda^2 J_1}{4}-\frac{J_1^3\lambda }{12}+\frac{J_1^5}{80}.
\end{align*}
This gives that
\begin{align}\label{est-21}
\sum_{k:\,\lambda_k<\lambda}(\lambda-\lambda_k)&\le   \frac{2n\omega(n)}{(2\pi)^{n}} |\Omega|^{1/2}|I(\Omega)|^{1/2}    \left[\frac{\lambda^2 J_1}{4}-\frac{J_1^3\lambda }{12}+\frac{J_1^5}{80}\right]\nonumber\\
&=\frac{\omega(n) |\Omega|}{(2\pi)^{n}} \frac{\lambda^2 }{2} \left[1- \frac{J_1^2 }{3\lambda}+\frac{J_1^4 }{20 \lambda^2}\right].
\end{align}
By \eqref{J-compare} and \eqref{first-eigenvalue}, we see that
\begin{align*}
\sum_{k:\,\lambda_k<\lambda}(\lambda-\lambda_k)&\le \frac{\omega(n) |\Omega|}{(2\pi)^{n}} \frac{\lambda^2 }{2} \left[1- \frac{J_1^2 }{3\lambda}+\frac{J_1^4 }{20 \lambda^2}\right]\le \frac{\omega(n) |\Omega|}{(2\pi)^{n}} \frac{\lambda^2 }{2} \left[1- \frac{J_1^2 }{3\lambda}+\frac{J_1^2\lambda_1 }{160 \lambda^2}\right]\\
&\le \frac{\omega(n) |\Omega|}{(2\pi)^{n}} \frac{\lambda^2 }{2} \left[1- \frac{157 J_1^2  }{480 \lambda}\right]
\end{align*}
for $\lambda>\lambda_1$, this proves Theorem \ref{Dirichlet-2} for the case $n=2$.

Let us consider the case $n\ge 3$.
%

Recall that
$$0\le -\psi'(r)\le 2(2\pi)^{-n}|\Omega|^{1/2}|I(\Omega)|^{1/2}.$$
The maximum of the integral
$$-\int_0^\infty n\omega(n)\left(\frac{\lambda r^n}{n} -\frac{r^{n+2}}{n+2}\right) \psi'(r)\,dr$$
attains for
$$
-\psi'(r)=\begin{cases}
2(2\pi)^{-n}|\Omega|^{1/2}|I(\Omega)|^{1/2},&  r\in [t_0,t_0+J_1]\\
0, & otherwise.
\end{cases}$$
So for $\psi(r)$, we have
$$
\psi(r)=\begin{cases}
(2\pi)^{-n}|\Omega|,& r\le t_0,\\
(2\pi)^{-n}|\Omega|-2(2\pi)^{-n}|\Omega|^{1/2}|I(\Omega)|^{1/2}(r-t_0),&  r\in [t_0,t_0+J_1]\\
0, & otherwise.
\end{cases}$$

By Theorem \ref{Dirichlet-1}, we have
\begin{align*}
\sum_{k:\,\lambda_k<\lambda}(\lambda-\lambda_k)&\le  \int_0^{t_0+J_1} n\omega(n)(\lambda-|r|^2)(2\pi)^{-n}|\Omega| r^{n-1} \,dr\nonumber\\
 &\quad-  \int_{t_0}^{t_0+J_1} n\omega(n)(\lambda-|r|^2)\left[2(2\pi)^{-n}|\Omega|^{1/2}|I(\Omega)|^{1/2}(r-t_0)\right] r^{n-1} \,dr.
\end{align*}

On the other hand,
recall that the function
$$\theta(r):=\frac{\lambda r^n}{n} -\frac{r^{n+2}}{n+2}$$
has  one maximum in $(0,\infty)$, which attains at the point
$r=\sqrt\lambda.$
Since $t_0$ satisfies $t_0<\sqrt\lambda$, $t_0+J_1>\sqrt\lambda$,
$$\frac{\lambda (n+2)}{n} [(t_0+J_1)^{n}-t_0^{n}] = (t_0+J_1)^{n+2}-t_0^{n+2},$$
equivalently,
$$\frac{\lambda t_0^n}{n} -\frac{t_0^{n+2}}{n+2}=\frac{\lambda (t_0+J_1)^n}{n} -\frac{(t_0+J_1)^{n+2}}{n+2}. $$
Note that
$$\theta'(r)=\lambda r^{n-1}-r^{n+1},$$
so around $\sqrt\lambda$, for $0<\delta<\sqrt\lambda$,
\begin{align*}
0>\frac{\theta'(\sqrt\lambda-\delta)}{\theta'(\sqrt\lambda+\delta)}=\frac{\lambda-(\sqrt\lambda-\delta)^2}{\lambda-(\sqrt\lambda+\delta)^2 }\frac{(\sqrt\lambda-\delta)^{n-1}}{(\sqrt\lambda+\delta)^{n-1}}=
-\frac{2\sqrt\lambda \delta-\delta^2}{2\sqrt\lambda \delta+\delta^2}\frac{(\sqrt\lambda-\delta)^{n-1}}{(\sqrt\lambda+\delta)^{n-1}}>-1.
\end{align*}
This implies that
$$t_0=\sqrt\lambda-\alpha J_1<\sqrt\lambda-\frac {J_1}{2},$$
i.e.,
$$\alpha>\frac 12.$$

From this  we further deduce that
\begin{align*}
\sum_{k:\,\lambda_k<\lambda}(\lambda-\lambda_k)&\le  \int_0^{\sqrt\lambda} n\omega(n)(\lambda-|r|^2)(2\pi)^{-n}|\Omega| r^{n-1} \,dr\nonumber\\
 &\quad-  \int_{\sqrt\lambda-J_1/2}^{\sqrt\lambda} n\omega(n)(\lambda-|r|^2)\left[2(2\pi)^{-n}|\Omega|^{1/2}|I(\Omega)|^{1/2}(r-\sqrt\lambda+J_1/2)\right] r^{n-1} \,dr\\
 &\le \frac{ 2\omega(n)}{(n+2)} \frac{|\Omega|}{(2\pi)^{n}}\lambda^{1+n/2}\\
 &\quad-\frac{2 n\omega(n)|\Omega|^{1/2}|I(\Omega)|^{1/2}}{(2\pi)^{n}}(\sqrt\lambda-\frac{J_1}2)^{n-1}(2\sqrt\lambda-\frac{J_1}{2})
 \int_{\sqrt\lambda-J_1/2}^{\sqrt\lambda} (\sqrt\lambda-r)(r-\sqrt\lambda +\frac{J_1}2)\,dr,
\end{align*}
where
\begin{align*}
\int_{\sqrt\lambda-J_1/2}^{\sqrt\lambda} (\sqrt\lambda-r)(r-\sqrt\lambda +\frac{J_1}2)\,dr&
= \left[\frac{1}{3} (\sqrt\lambda-r)^3-\frac{J_1}4  (\sqrt\lambda-r)^2\right]\bigg|_{\sqrt\lambda-J_1/2}^{\sqrt\lambda}
=\frac{5}{48}J_1^3.
\end{align*}
Thus we conclude that
\begin{align*}
\sum_{k:\,\lambda_k<\lambda}(\lambda-\lambda_k)&\le \frac{ 2\omega(n)}{(n+2)} \frac{|\Omega|}{(2\pi)^{n}}\lambda^{1+n/2}\\
 &\quad-\frac{2 n\omega(n)|\Omega|^{1/2}|I(\Omega)|^{1/2}}{(2\pi)^{n}}(\sqrt\lambda-\frac{J_1}2)^{n-1}(2\sqrt\lambda-\frac{J_1}{2})
 \frac{5}{48}J_1^3\\
 &=  \frac{ 2\omega(n)}{n+2} \frac{|\Omega|}{(2\pi)^{n}}\lambda^{1+n/2}-\frac{ 5n\omega(n)|\Omega|}{48(2\pi)^{n}}(\sqrt\lambda-\frac{J_1}2)^{n-1}(2\sqrt\lambda-\frac{J_1}{2})J_1^2
\end{align*}

By \eqref{J-compare} and \eqref{first-eigenvalue},
\begin{align*}
J_1\le \frac{\sqrt{n+2}}{2R_\Omega\sqrt n}
\end{align*}
and
\begin{equation*}
\lambda_1 \ge  \frac{(2\pi)^2}{|\Omega|^{2/n}\omega(n)^{2/n}}=\frac{(2\pi)^2}{R_\Omega^2\omega(n)^{4/n}}= \frac{4 (\Gamma(1+\frac n2))^{4/n}}{R_\Omega^2},
\end{equation*}
we then have for $\lambda>\lambda_1$,
\begin{align*}
(\sqrt\lambda-\frac{J_1}2)^{n-1}\ge \lambda^{\frac{n-1}{2}}\left(1-\frac{\sqrt{n+2}}{8\sqrt n \Gamma(1+\frac n2))^{2/n}}\right)^{n-1},
\end{align*}
and
$$2\sqrt\lambda-\frac{J_1}{2}\ge \sqrt\lambda\left(2-\frac{\sqrt{n+2}}{8\sqrt n \Gamma(1+\frac n2))^{2/n}}\right).$$
Hence, we finally deduce that
\begin{align*}
\sum_{k:\,\lambda_k<\lambda}(\lambda-\lambda_k)&\le   \frac{ 2\omega(n)}{n+2} \frac{|\Omega|}{(2\pi)^{n}}\lambda^{1+n/2}\left(1-C_1(n)J_1^2\lambda^{-1}\right),
\end{align*}
where
$$C_1(n)=\frac{5n(n+2)}{96}\left(1-\frac{\sqrt{n+2}}{8\sqrt n \Gamma(1+\frac n2))^{2/n}}\right)^{n-1}\left(2-\frac{\sqrt{n+2}}{8\sqrt n \Gamma(1+\frac n2))^{2/n}}\right).$$
The proof for $n\ge 3$ is complete.
\end{proof}
\begin{rem}\label{remark-1}\rm
By  Stirling's formula
$$\Gamma(1+x)\ge \sqrt{2\pi x}(x/e)^x,$$
$\Gamma(1+\frac n2))^{2/n}$ behaves as $\frac{n}{2e}$ along $n\to\infty$.
This implies that
$$\left(1-\frac{\sqrt{n+2}}{8\sqrt n \Gamma(1+\frac n2))^{2/n}}\right)^{n-1}\to e^{-e/4} \approx 0.507.$$
So we can find a uniformly positive lower bound  for $\left(1-\frac{\sqrt{n+2}}{8\sqrt n \Gamma(1+\frac n2))^{2/n}}\right)^{n-1}\left(2-\frac{\sqrt{n+2}}{8\sqrt n \Gamma(1+\frac n2))^{2/n}}\right)$
for all $n\ge 3$. In fact, a calculation shows that the quantity
$$\left(1-\frac{\sqrt{n+2}}{8\sqrt n \Gamma(1+\frac n2))^{2/n}}\right)^{n-1}\left(2-\frac{\sqrt{n+2}}{8\sqrt n \Gamma(1+\frac n2))^{2/n}}\right)\approx 1.402, 1.341,1.248, 1.045, 1.016, 1.014$$
for $n=3,4,5, 100, 2000, 100000$.
\end{rem}

By letting $\lambda=(1+2/n)\lambda_k$ and dividing both terms in \eqref{riesz-dirichlet}  we have the following improvement
on eigenvalues from \cite{LY83}, see also \cite{La97}.
\begin{cor}\label{eigenvalue-dirichlet-1}
Let $\Omega$ be a bounded open set in $\rn$, $n\ge 2$. Then there exists $C_1(n)>0$ such that it holds
for all $k\ge 1$ that
\begin{align}\label{riesz-dirichlet-1}
k&\le \left(1+\frac 2n\right)^{n/2}\lambda_k^{\frac n2} \frac{|\Omega| |\omega(n)|}{(2\pi)^{n}} \left(1 -\frac{C_1(n)n}{4(n+2)\lambda_k} \frac{ |\Omega|}{|I(\Omega)|}\right).
\end{align}
\end{cor}
Recall that for $n=2$,
$$C_1(2)=\frac{157}{480},$$
\eqref{riesz-dirichlet-1} is equivalent to
\begin{align*}
\lambda_k \ge  \frac{(2\pi)^{2}}{2|\Omega| |\omega(2)|}  k+ \frac{C_1(2)}{8} \frac{ |\Omega|}{|I(\Omega)|}.
\end{align*}
Obviously, $\frac{C_1(2)}{8}$ is close to $1/24$ and is larger than $1/32$ from \eqref{melas-sum} (cf. \cite{CQW13,{KVW09}}).

By Theorem \ref{Dirichlet-2}, we derive the following qualitative results.
\begin{thm}\label{infinite-dirichlet-polya}
Let $\Omega$ be a bounded open set in $\rn$, $n\ge 2$.
For any $\lambda_{j}>\lambda_1$, let $\lambda_{j_1}$ be the smallest eigenvalue that satisfies
$$\lambda_{j_1}\ge \left(\frac{2\lambda_{j}^{1+\frac n2}}{C_1(n)|J_1|^2}\right)^{2/n}.$$
Then there exists an eigenvalue $\lambda_{k_j}\in [\lambda_{j}, \lambda_{j_1}]$, that satisfies
$$\frac{|\Omega| |\omega(n)|}{(2\pi)^{n}}  \lambda_{k_{j}}^{\frac n2}\ge k_{j}-1+\frac{|\Omega| |\omega(n)|}{(2\pi)^{n}} \frac{nC_1(n)|J_1|^2}{2(n+2)}\lambda_{k_{j}}^{\frac n2-1}. $$
Above $C_1(n)$ is given by \eqref{defn-c1}.
\end{thm}
\begin{proof}
Let $\{k_j\}_{j\ge 1}$ be a subset of $\mathbb{N}$ such that
$$0<\lambda_1=\cdots=\lambda_{k_1}<\lambda_{k_1+1}=\cdots=\lambda_{k_2}<\lambda_{k_2+1}\cdots. $$
Let us set also
$$F_D(\lambda):= \frac{2}{n+2} \lambda^{1+\frac n2} \frac{|\Omega| |\omega(n)|}{(2\pi)^{n}}\left(1 -\frac{C_1(n)|J_1|^2}{2\lambda}\right)_+-\sum_{k:\,\lambda_k<\lambda}(\lambda-\lambda_k).$$
Obviously, $F_D(\lambda)$ is absolutely continuous on $(0,\infty)$, and smooth in each interval $ (\lambda_{k_{j}},\lambda_{k_{j}+1})$.
Since $\lambda_1>8|J_1|^2$ by \eqref{first-j1}, for $\lambda>\lambda_1$,
$$F_D(\lambda)= \frac{2}{n+2} \lambda^{1+\frac n2} \frac{|\Omega| |\omega(n)|}{(2\pi)^{n}}\left(1 -\frac{C_1(n)|J_1|^2}{2\lambda}\right)-\sum_{k:\,\lambda_k<\lambda}(\lambda-\lambda_k).$$
For $\lambda\in (\lambda_{k_{j}},\lambda_{k_{j}+1})=(\lambda_{k_{j}},\lambda_{k_{j+1}})$, it holds that
$$F_D(\lambda)= \frac{2}{n+2} \lambda^{1+\frac n2} \frac{|\Omega| |\omega(n)|}{(2\pi)^{n}}\left(1 -\frac{C_1(n) |J_1|^2}{2\lambda}\right)-\sum_{i=1}^{j}(k_{i}-k_{i-1})(\lambda-\lambda_{k_{j}}),$$
where we set $k_0=0$.
Moreover, for $\lambda\in (\lambda_{k_{j}},\lambda_{k_{j}+1})$, $j\ge 1$,
\begin{align}\label{F-increasing}
F'_D(\lambda)&=\frac{|\Omega| |\omega(n)|}{(2\pi)^{n}}  \lambda^{\frac n2}\left(1 -\frac{nC_1(n)|J_1|^2}{2(n+2)\lambda }\right)-k_{j}\nonumber\\
&=\frac{|\Omega| |\omega(n)|}{(2\pi)^{n}}  \lambda^{\frac n2-1}\left(\lambda -\frac{nC_1(n)|J_1|^2}{2(n+2) }\right)-k_{j},
\end{align}
which implies that $F'_D(\lambda)$ is increasing on each $(\lambda_{k_{j}},\lambda_{k_{j}+1})$ since   $\lambda_1>8|J_1|^2$.

By Theorem \ref{Dirichlet-2}, we see that for $\lambda>\lambda_1$, it holds that
\begin{align}
F_D(\lambda)= \frac{2}{n+2} \lambda^{1+\frac n2} \frac{|\Omega| |\omega(n)|}{(2\pi)^{n}}\left(1 -\frac{C_1(n)|J_1|^2}{2\lambda}\right)-\sum_{k:\,\lambda_k<\lambda}(\lambda-\lambda_k)>  \frac{C_1(n)|J_1|^2}{n+2} \lambda^{\frac n2} \frac{|\Omega| |\omega(n)|}{(2\pi)^{n}}.
\end{align}
For a given $\lambda_k$, it equals $\lambda_{k_{j_0}}$ for some $k_{j_0}$.
Let $\lambda_{k_{j_1}+1}$ be the smallest eigenvalue that satisfies
$$\lambda_{k_{j_1}+1}\ge \left(\frac{2\lambda_{k_{j_0}}^{1+\frac n2}}{C_1(n)|J_1|^2}\right)^{2/n}.$$
The choice of $ \lambda_{k_{j_0}}$ and $\lambda_{k_{j_1}+1}$ implies  that
$$F_D(\lambda_{k_{j_1}+1})>F_D(\lambda_{k_{j_0}}),$$
and hence,
$$F_D(\lambda_{k_{j_1}+1})-F_D(\lambda_{k_{j_0}})=\int_{\lambda_{k_{j_0}}}^{\lambda_{k_{j_0}+1}} F'_D(\lambda)\,d\lambda=\sum_{j_0\le j\le j_1}\int_{\lambda_{k_{j}}}^{\lambda_{k_{j}+1}} F'_D(\lambda)\,d\lambda>0.$$
This together the piecewise smoothness of $F'_D$ implies that $F_D$ is increasing at least on part of an interval $( \lambda_{k_{j}}, \lambda_{k_{j}+1})$,
and consequently, $F'_D(\lambda)$ is positive on part of it. Since $F'_D(\lambda)$ is increasing on each $(\lambda_{k_{i}},\lambda_{k_{i}+1})$,
we see that
$$ F_D'(\lambda_{k_{j}+1}-\epsilon)=\frac{|\Omega| |\omega(n)|}{(2\pi)^{n}}  (\lambda_{k_{j}+1}-\epsilon)^{\frac n2}\left(1-\frac{nC_1(n)|J_1|^2}{2(n+2)(\lambda_{k_{j}+1}-\epsilon)}\right)-k_{j}>0$$
for some sufficiently small $\epsilon>0$.
Sending $\epsilon$ to zero gives
$$\frac{|\Omega| |\omega(n)|}{(2\pi)^{n}}  \lambda_{k_{j}+1}^{\frac n2}\ge k_{j}+\frac{|\Omega| |\omega(n)|}{(2\pi)^{n}} \frac{nC_1(n)|J_1|^2}{2(n+2)}\lambda_{k_{j}+1}^{\frac n2-1}. $$
The proof is complete.
\end{proof}
\begin{proof}[Proof of Corollary \ref{cor-infinite-polya}]
For $n\ge 3$, $\lambda_{k_j+1}^{\frac n2-1}$ tends to infinity as $k_j\to \infty$, and the conclusion follows from Theorem \ref{infinite-dirichlet-polya} immediately.
\end{proof}

\section{Improved Kr\"oger inequality}
\hskip\parindent  In this section, we provide proofs for Theorem \ref{Neumann}.
We assume that $-\Delta_N$ on $\Omega$ has discrete spectrum so that
$$0=\gamma_1\le \gamma_2\le \cdots.$$
Let $u_k$ be the corresponding normalized eigenfunctions to $\gamma_k$. Then for each $\xi\in\rn$ we have
\begin{align*}
\sum_{k=1}^\infty |\hat{u}_k(\xi)|^2=(2\pi)^{-n} |\Omega|
\end{align*}
and
\begin{align*}
\sum_{k=1}^\infty \gamma_k|\hat{u}_k(\xi)|^2=(2\pi)^{-n} |\xi|^2|\Omega|.
\end{align*}
The two identity together implies that
\begin{align}\label{lower-neumann-function}
\sum_{k:\,\gamma_k<\gamma} (\gamma-\gamma_k)|\hat{u}_k(\xi)|^2\ge
\begin{cases}
(2\pi)^{-n}(\gamma-|\xi|^2)|\Omega|, & |\xi|<\sqrt\gamma,\\
0, & |\xi|\ge \sqrt\gamma.
\end{cases}
\end{align}

We can now prove Theorem \ref{Neumann}.
\begin{proof}[Proof of Theorem \ref{Neumann}]
To estimate
$$\sum_{k:\,\gamma_k<\gamma} (\gamma-\gamma_k)=\sum_{k:\,\gamma_k<\gamma} \int_{\rn}(\gamma-\gamma_k)|\hat{u}_k(\xi)|^2\,d\xi,$$
let us denote by $F_\gamma(\xi)$ the quantity
$$F_\gamma(\xi)=\left(\sum_{k:\,\gamma_k<\gamma} (\gamma-\gamma_k)|\hat{u}_k(\xi)|^2\right)^{1/2}.$$

The function $F_\gamma(\xi)$ is Lipschitz continuous on $\rn$ which satisfies
\begin{align*}
|\nabla F_\gamma(\xi)|&\le \frac{1}{F_\gamma(\xi)} \sum_{k:\,\gamma_k<\gamma} (\gamma-\gamma_k) |\hat{u}_k(\xi)| \frac{1}{(2\pi)^{n/2}}\left|\int_\Omega xe^{-ix\cdot\xi}u_k(x)\,dx\right| \\
&\le \left(\sum_{k:\,\gamma_k<\gamma} (\gamma-\gamma_k) \frac{1}{(2\pi)^{n}}\left|\int_\Omega xe^{-ix\cdot\xi}u_k(x)\,dx\right|^2\right)^{1/2}\\
&\le \sqrt{\gamma}  \left( \frac{1}{(2\pi)^{n}}\int_\Omega x^2\,dx\right)^{1/2}.
\end{align*}
By suitably moving the position of $\Omega$, we see that it further holds
\begin{align*}
|\nabla F_\gamma(\xi)|&\le \frac{\sqrt\gamma}{(2\pi)^{n/2}}\left(\inf_{x_0\in\rn}\int_{\Omega}|x-x_0|^2\,dx\right)^{1/2}= \frac{\sqrt\gamma}{(2\pi)^{n/2}}I(\Omega)^{1/2}.
\end{align*}

Let us denote by $F_\gamma^\ast$ the Schwarz symmetrization of $F_\gamma$.
Lemma \ref{lem-1} tells us that
\begin{align*}
|\nabla F_\gamma^\ast(\xi)|&\le \||\nabla F_\gamma^\ast\|_{L^\infty(\rn)} \le  \frac{\sqrt\gamma}{(2\pi)^{n/2}}I(\Omega)^{1/2}.
\end{align*}

On the other hand, note that it follows from \eqref{lower-neumann-function} that
$$F_\gamma(\xi)\ge \frac{|\Omega|^{1/2}}{(2\pi)^{n/2}}\sqrt{(\gamma-|\xi|^2)_+},$$
which implies that
$$F_\gamma^\ast(\xi)\ge \frac{|\Omega|^{1/2}}{(2\pi)^{n/2}}\sqrt{(\gamma-|\xi|^2)_+}. $$

Set
$$\varphi_\gamma(r):=\frac{|\Omega|^{1/2}}{(2\pi)^{n/2}}\sqrt{(\gamma-r^2)_+}.$$
Then for $r<\sqrt\gamma$,
$$\varphi'_\gamma(r)=\frac{|\Omega|^{1/2}}{(2\pi)^{n/2}}\frac{-r}{\sqrt{\gamma-r^2}},$$
which is decreasing on $(0,\sqrt\gamma)$ and tends to $-\infty$ as $r\to\sqrt\gamma$.

Consider the equation
$$\frac{|\Omega|^{1/2}}{(2\pi)^{n/2}}\frac{r}{\sqrt{\gamma-r^2}}=\frac{\sqrt\gamma}{(2\pi)^{n/2}}I(\Omega)^{1/2},$$
which has a unique positive solution
$$r_0=\frac{\gamma}{\sqrt{4J_1^2+\gamma}}.$$

Since
$$F_\gamma^\ast(\xi)=F_\gamma^\ast(|\xi|)\ge \varphi_\gamma(|\xi|), \,\&\, |\nabla F_\gamma^\ast(\xi)|\le \frac{\sqrt\gamma}{(2\pi)^{n/2}}I(\Omega)^{1/2},$$
we find that
\begin{align}
F_\gamma^\ast(\xi)\ge &
\begin{cases}
\frac{|\Omega|^{1/2}}{(2\pi)^{n/2}}\sqrt{(\gamma-|\xi|^2)}, &\, |\xi|\le r_0,\\
\frac{|\Omega|^{1/2}}{(2\pi)^{n/2}}\frac{2J_1\sqrt\gamma}{\sqrt{4J_1^2+\gamma}}-\frac{\sqrt\gamma}{(2\pi)^{n/2}}I(\Omega)^{1/2} (|\xi|-r_0), & \, r_0<|\xi|<r_0+\frac{4J_1^2}{\sqrt{4J_1^2+\gamma}},\\
0,&\, |\xi|\ge r_0+\frac{4J_1^2}{\sqrt{4J_1^2+\gamma}}.
\end{cases}
\end{align}
Note that
$$r_0+\frac{4J_1^2}{\sqrt{4J_1^2+\gamma}}=\sqrt{4J_1^2+\gamma}=:r_1,$$
and so for $r_0<|\xi|<r_1$,
\begin{align*}
F_\gamma^\ast(\xi)&\ge \frac{|\Omega|^{1/2}}{(2\pi)^{n/2}}\frac{2J_1\sqrt\gamma}{\sqrt{4J_1^2+\gamma}}-\frac{\sqrt\gamma}{(2\pi)^{n/2}}I(\Omega)^{1/2} (|\xi|-r_0)\\
&=\frac{\sqrt\gamma}{(2\pi)^{n/2}}I(\Omega)^{1/2} (r_1-|\xi|)\\
&\ge \frac{|\Omega|^{1/2}}{(2\pi)^{n/2}}\sqrt{(\gamma-|\xi|^2)_+}.
\end{align*}

We therefore conclude that
\begin{align*}
\sum_{k:\,\gamma_k<\gamma} (\gamma-\gamma_k)&=\sum_{k:\,\gamma_k<\gamma} \int_{\rn}(\gamma-\gamma_k)|\hat{u}_k(\xi)|^2\,d\xi\\
&=\int_{\rn} [F_\gamma(\xi)]^2\,d\xi=\int_{\rn}  [F_\gamma^\ast(\xi)]^2\,d\xi\\
&=\int_0^\infty   {n} {\omega(n) } [F_\gamma^\ast(r)]^2 r^{n-1}\,dr\\
&\ge \int_0^{r_0}  {n}{\omega(n) r^{n-1} } \frac{|\Omega|}{(2\pi)^{n}}{(\gamma-r^2)}\,dr \\
&\quad+\int_{r_0}^{r_1} {n}{\omega(n) r^{n-1} } \frac{|I(\Omega)|\gamma}{(2\pi)^{n}}\left(r_1-r\right)^2\,dr\\
&\ge \int_0^{\sqrt\gamma}  {n}{\omega(n) r^{n-1} } \frac{|\Omega|}{(2\pi)^{n}}{(\gamma-r^2)}\,dr\\
&\quad+\int_{\sqrt\gamma}^{r_1} {n}{\omega(n) r^{n-1} } \frac{|I(\Omega)|\gamma}{(2\pi)^{n}}\left(r_1-r\right)^2\,dr.
\end{align*}

For the first integral, we have
\begin{align*}
\int_0^{\sqrt\gamma}  {n}{\omega(n) r^{n-1} } \frac{|\Omega|}{(2\pi)^{n}}{(\lambda-r^2)}\,dr&=\frac{\omega(n)|\Omega|}{(2\pi)^{n}} \frac{2\gamma^{1+\frac n2}}{n+2}.
\end{align*}
For the second integral, it holds that
\begin{align*}
\int_{\sqrt\gamma}^{r_1} {n}{\omega(n) r^{n-1} } \frac{|I(\Omega)|\gamma}{(2\pi)^{n}}\left(r_1-r\right)^2\,dr
&\ge \frac{n|I(\Omega)|\omega(n) \gamma}{3(2\pi)^{n}} \gamma^{\frac{n-1}{2}} \left(r_1-\sqrt\gamma\right)^3\\
&\ge  \frac{n|I(\Omega)|\omega(n) \gamma}{3(2\pi)^{n}} \gamma^{\frac{n-1}{2}}  \left(\frac{4J_1^2}{\sqrt{4J_1^2+\gamma}+\sqrt\gamma}\right)^3\\
&\ge \frac{n|I(\Omega)|\omega(n) }{3(2\pi)^{n}} \gamma^{\frac{n+1}{2}} \frac{8J_1^6}{(4J_1^2+\gamma)^{3/2}} \\
&= \frac{n|\Omega|\omega(n) }{3(2\pi)^{n}} \gamma^{\frac{n+1}{2}} \frac{2J_1^4}{(4J_1^2+\gamma)^{3/2}}
\end{align*}
Summing up the above two estimates and using the definition of $r_0,r_1$ we arrive at
\begin{align*}
\sum_{k:\,\gamma_k<\gamma} (\gamma-\gamma_k)&\ge \frac{\omega(n)|\Omega|}{(2\pi)^{n}} \frac{2\gamma^{1+\frac n2}}{n+2}+  \frac{n|\Omega|\omega(n) }{3(2\pi)^{n}} \gamma^{\frac{n+1}{2}} \frac{2J_1^4}{(4J_1^2+\gamma)^{3/2}}\\
&=\frac{\omega(n)|\Omega|}{(2\pi)^{n}} \frac{2\gamma^{1+\frac n2}}{n+2}\left(1+\frac{n(n+2) }{3 } \frac{J_1^4}{\gamma^{1/2} (4J_1^2+\gamma)^{3/2}}\right).
\end{align*}
The proof is complete.
\end{proof}

\begin{thm}\label{thm-infinite-neumann}
Let $\Omega$ be a bounded open set in $\rn$, $n\ge 3$. Suppose that $-\Delta_N$ has a discrete spectrum
$$0= \gamma_1\le \gamma_2\le \gamma_3\le \cdots.$$
For any given $\gamma_j\ge 4|J_1|^2$, let $\gamma_{j_1}$ be the smallest eigenvalue satisfying
$$\gamma_{{j_1}}\ge \left( j\frac{(2\pi)^{n}}{\omega(n)|\Omega|}\right)^{\frac{2(n+2)}{n(n-2)}}  \left( 3\left(\frac{n+2}{2}\right)^{2/n} \frac{2^{3/2}}{nJ_1^4}\right)^{\frac{2}{n-2}}.$$
 there exists $\gamma_{k_j} \in [\gamma_j, \gamma_{j_1}]$ satisfying
$$ k_{j}- \frac{J_1^4}{2^{3/2}} \frac{\omega(n)|\Omega|}{(2\pi)^{n}}\frac{n}{3}\frac{n-2}{2}\gamma_{k_{j}}^{\frac{n}{2}-2} \ge \frac{\omega(n)|\Omega|}{(2\pi)^{n}} \gamma_{k_{j}}^{n/2}. $$
\end{thm}
\begin{proof}
Let $\{k_j\}_{j\ge 1}$ be a subset of $\mathbb{N}$ such that
$$0=\gamma_1=\cdots=\gamma_{k_1}<\gamma_{k_1+1}=\cdots=\gamma_{k_2}<\gamma_{k_2+1}=\cdots=\gamma_{k_3}<\cdots,$$
and $k_0=0$. 
For $\gamma>4|J_1|^2$, set
$$F_N(\gamma):= \sum_{k:\,\gamma_k<\gamma} (\gamma-\gamma_k) -\frac{\omega(n)|\Omega|}{(2\pi)^{n}} \frac{2\gamma^{1+\frac n2}}{n+2}\left(1+\frac{n(n+2) }{6} \frac{J_1^4}{2^{3/2}}\gamma^{-2} \right).$$
Assume that $\gamma_{k_{j_0}}>4|J_1|^2$. For $j_1>j_0$, $\gamma\in (\gamma_{k_{j_1}},\gamma_{k_{j_1}+1})=(\gamma_{k_{j_1}},\gamma_{k_{{j_1}+1}})$, we have
\begin{align}
F_N(\gamma)= \sum_{j=1}^{j_1} (k_j-k_{j-1}) (\gamma-\gamma_{k_j}) -\frac{\omega(n)|\Omega|}{(2\pi)^{n}} \frac{2\gamma^{1+\frac n2}}{n+2}\left(1+\frac{n(n+2) }{6} \frac{J_1^4}{2^{3/2}}\gamma^{-2} \right),
\end{align}
and hence,
\begin{align}\label{decreasing-fn}
F'_N(\gamma)=
k_{j_1} -\frac{\omega(n)|\Omega|}{(2\pi)^{n}} \gamma^{n/2} -  \frac{J_1^4}{2^{3/2}} \frac{\omega(n)|\Omega|}{(2\pi)^{n}}\frac{n}{3}\frac{n-2}{2}\gamma^{\frac{n}{2}-2}.
\end{align}
Obviously, $F_N'(\gamma)$ decreases on $(\gamma_{k_{j_1}},\gamma_{k_{j_1}+1})$ as $n\ge 3$.

By Theorem \ref{Neumann},
$$F_N(\gamma)>\frac{\omega(n)|\Omega|}{(2\pi)^{n}} \frac{n\gamma^{\frac n2-1}}{3}  \frac{J_1^4}{2^{3/2}} $$
for $\gamma>4|J_1|^2$.
For a given $\gamma_j>4|J_1|^2$, let $\gamma_{k_{j_0}}=\gamma_j$, and $\gamma_{k_{j_1}+1}$ be the smallest eigenvalue satisfying
$$\frac{\omega(n)|\Omega|}{(2\pi)^{n}} \frac{n\gamma_{k_{j_1}+1}^{\frac n2-1}}{3} \frac{J_1^4}{2^{3/2}}\ge j^{1+2/n}\left(\frac{n+2}{2}\right)^{2/n} \frac{(2\pi)^2}{ |\Omega|^{2/n}\omega(n)^{2/n}},$$
which is equivalent to
$$\gamma_{k_{j_1}+1}\ge \left( j\frac{(2\pi)^{n}}{\omega(n)|\Omega|}\right)^{\frac{2(n+2)}{n(n-2)}}  \left( 3\left(\frac{n+2}{2}\right)^{2/n} \frac{2^{3/2}}{nJ_1^4}\right)^{\frac{2}{n-2}}.$$

By \eqref{kroger-eigenvalue} together with Theorem \ref{Neumann}, we  have
$$F_N(\gamma_{k_{j_1}+1})>j\gamma_j>F_N(\gamma_{j})=F_N(\gamma_{k_{j_0}}).$$
This implies that there exists $(\gamma_{k_j},\gamma_{k_j+1})\subset (\gamma_{k_{j_0}},\gamma_{k_{j_1}+1})$ such that it contains a small interval that $F_N(\gamma)$ is increasing.
By \eqref{decreasing-fn}, this means there exists $\epsilon_j>0$ such that
 $F_N(\gamma)$ is increasing on $(\gamma_{k_{j}}, \gamma_{k_{j}}+\epsilon_j),$
 and consequently,
$$ k_{j} -\frac{\omega(n)|\Omega|}{(2\pi)^{n}} \gamma_{k_{j_1}}^{n/2} -  \frac{J_1^4}{2^{3/2}} \frac{\omega(n)|\Omega|}{(2\pi)^{n}}\frac{n}{3}\frac{n-2}{2}\gamma_{k_{j_1}}^{\frac{n}{2}-2} \ge 0.$$
This is equivalent to
$$ k_{j_1}- \frac{J_1^4}{2^{3/2}} \frac{\omega(n)|\Omega|}{(2\pi)^{n}}\frac{n}{3}\frac{n-2}{2}\gamma_{k_{j_1}}^{\frac{n}{2}-2} \ge \frac{\omega(n)|\Omega|}{(2\pi)^{n}} \gamma_{k_{j_1}}^{n/2},$$
which completes the proof.
\end{proof}
\begin{proof}[Proof of Corollary \ref{cor-infinite-neumann}]
If $n\ge 5$, then $\gamma_{k_j}^{\frac{n}{2}-2}\to \infty$ as $k_j\to\infty$, the conclusion follows from Theorem \ref{thm-infinite-neumann}.
\end{proof}

\subsection*{Acknowledgments}
\addcontentsline{toc}{section}{Acknowledgments} \hskip\parindent
R. Jiang was partially supported by NNSF of China (12471094 \& 11922114),  F.H. Lin
was in part supported by the National Science Foundation Grant  DMS2247773 and DMS1955249.

\noindent Zaihui Gan \\
\noindent Center for Applied Mathematics\\
\noindent Tianjin University\\
\noindent Tianjin 300072, China\\
\noindent ganzaihui2008cn@tju.edu.cn

\

\noindent Renjin Jiang \\
\noindent Academy for Multidisciplinary Studies\\
\noindent Capital Normal University\\
\noindent Beijing 100048, China\\
\noindent {rejiang@cnu.edu.cn}

\

\noindent Fanghua Lin\\
\noindent Courant Institute of Mathematical Sciences\\
\noindent 251 Mercer Street, New York, NY 10012, USA \\
\noindent linf@cims.nyu.edu

\end{document}